\documentclass[11pt]{article}

\usepackage{epsf,amssymb,latexsym,amsmath}
\usepackage{graphicx}
\usepackage{tikz}
\usetikzlibrary{arrows}

\newcommand{\proof}{\noindent{\bf Proof.\ }}
\newcommand{\qed}{\hfill $\square$ \bigskip}

\newtheorem{theorem}{\bf Theorem}[section]

\newtheorem{lemma}[theorem]{\bf Lemma}

\newtheorem{conjecture}[theorem]{\bf Conjecture}

\def\max {\mathop{\rm max}\nolimits}

\begin{document}

\title{Hamiltonicity of Cartesian products of trees with odd paths }

\author{
    Irena Hrastnik Ladinek\thanks{University of Maribor, FME, Maribor, Slovenia.} \and 
     Tja\v sa Paj Erker\footnotemark[1] \and
   Simon \v Spacapan\footnotemark[1] \footnote{IMFM, Jadranska 19, 1000 Ljubljana.} \footnote{The author is supported by ARIS program P1-0297.}}
\date{\today}

\maketitle

\begin{abstract}
 A  $\{P_2,P_3\}$-factor in a graph $G$ is a factor of $G$ in which every component  is a path on   two or three vertices. Let $T\Box P_n$ be the Cartesian product of a tree $T$ and a path on $n$ vertices. 
Kao and Weng \cite{kao} proved  that $T\Box P_n$ is hamiltonian if $T$ has a path factor  and $n$ is a sufficiently large even integer. 
In this article we prove that, for every odd $n$, there exists a tree $T$ of maximum degree 4 that has a $\{P_2,P_3\}$-factor  such that $T\Box P_n$ is not hamiltonian, 
thereby refuting a conjecture given in \cite{kao}.

\end{abstract}

\noindent
{\bf Key words}:  Cartesian product, hamiltonian graph.

\bigskip\noindent
{\bf AMS subject classification (2020)}: 05C45,  05C76.


\bigskip

\section{Introduction}
There exist several degree conditions which guarantee hamiltonicity of a graph. One of the most widely known is Dirac's condition \cite{dirac}, which  states that 
a graph of order $n$ and minimum degree at least $n/2$ is hamiltonian. Several generalizations of this result were obtained later; see \cite{bondy,chavatal,fan,ore}. 

Contrary to the (rather natural and intuitive)  restrictions on minimum degree of a graph, 
 hamiltonicity  of Cartesian products 
depends on the maximum degrees and the orders of  factors.  Batagelj and Pisanski \cite{batagelj} characterized hamiltonian products $T\Box C_n$, where $T$
is a tree and $C_n$ is a cycle on $n$ vertices. 

\begin{theorem}\cite{batagelj} \label{batagelj}
Let $T$ be a tree of maximum degree $\Delta$. Then $T\Box C_n$ is hamiltonian if and only if $n\geq \Delta$. 
\end{theorem}

For  product of a tree $T$ with an even path $P_n$ a   
 sufficient condition for hamiltonicity of  $T \Box P_n$   is given in \cite{kao}. 
If $T\Box P_n$ is hamiltonian, then $T$ has a $\{P_2,P_3\}$-factor (see \cite{spa}), and hence any sufficient condition 
for hamiltonicity of $T\Box P_n$ must assume that $T$ has a $\{P_2,P_3\}$-factor. 
 The condition given in \cite{kao} bounds maximum degree $\Delta(T)$  in terms of $n$.

\begin {theorem}  \cite{kao}\label{kitajci}
For any tree  $T$  the Cartesian product $T \Box P_n$ is hamiltonian if one of the following holds:
\begin{itemize}
\item[(a)] $T$ has a perfect matching and $n\geq \Delta(T)$
\item[(b)] $T$ has a  $\{P_2,P_3\}$-factor, $n$ is an even integer, and $n\geq 4\Delta (T)-2$.
\end{itemize}
\end{theorem}

In \cite{spa} the authors proved that the bound  $n\geq 4\Delta (T)-2$ in (b) is best possible. 
Both theorems above assert the following:  if maximum degree of one factor of $G\Box H$ is sufficiently small compared to the order of the other factor,  
 then $G\Box H$ is hamiltonian. It is interesting that  the converse is also  tacitly true --- as follows from Theorem \ref{cucek}.

Let $\nabla(G)$ be the minimum $k$ such that $G$ has a spanning tree of maximum degree $k$. In \cite{dimakopoulos} (see also \cite{klavzar}) the authors have proven  the following. 

\begin{theorem}\cite{dimakopoulos}\label{cucek}
Let $G$ and $H$ be connected graphs. If $G\Box H$ is hamiltonian, then $\nabla(G)\leq |V(H)|+1$ and $\nabla(H)\leq |V(G)|+1$. 
\end{theorem}
Hence hamiltonian   products $G\Box H$ have the property that the maximum degree of a spanning tree of $G$ is bounded by the order of $H$ (and vice versa).

 The proof of Theorem \ref{batagelj}
is based on the fact that $\{x\}\Box C_n$, which is isomorphic to $C_n$,   is a hamiltonian graph. To construct a Hamilton   cycle in $T\Box C_n$ we first construct a Hamilton  cycle in 
each $\{x\}\Box C_n$, where $x\in V(T)$. If $\Delta(T)\leq n $ we can then connect together all these cycles  to obtain a Hamilton cycle in 
$T\Box C_n$.  

If  $T$ has a  $\{P_2,P_3\}$-factorr and $n$ is even, we use a similar idea  to prove that $T\Box P_n$ is hamiltonian. 
Here it is  essential that both $P_2\Box P_n$ and $P_3\Box P_n$ are hamiltonian graphs, and 
to construct a Hamilton cycle in $T\Box P_n$ we start with a  $\{P_2,P_3\}$-factor $\displaystyle \cup G_i$  of $T$, where each $G_i$ is a $P_2$ or a $P_3$, and then we construct a Hamilton  cycle in each $G_i\Box P_n$. If $\Delta(T)$ is sufficiently small we can then connect all these cycles together to obtain a Hamilton cycle in $T\Box P_n$.

This approach does not work when $n$ is odd, for the simple reason that $P_3\Box P_n$ is not hamiltonian. As we shall see hamiltonicity of $G\Box P_n$ does not depend on 
maximum degree of $G$, contrary to the conjecture raised in \cite{kao}.

\begin{conjecture} \cite{kao}
Let $G$ be a graph with a  $\{P_2,P_3\}$-factor and $n\geq  4\Delta (G)-2$. If $G\Box P_n$ is balanced bipartite, then $G\Box P_n$ is hamiltonian. 
\end{conjecture}

In particular, if $n$ is odd, $T\Box P_n$ might be non-hamiltonian even if $n$ is arbitrarily large relative to $\Delta(T)$. 
In this article we refute the above conjecture by proving that for every odd integer $n$ there exists a graph $G$ of maximum degree four, that has  $\{P_2,P_3\}$-factor, such that for every odd  $i\leq n$ the graph $G\Box P_i$ is not hamiltonian.

\begin{theorem}\label{main}
For every positive   integer $n$    there exists a  balanced bipartite graph $G$  of    maximum degree four, that has a  $\{P_2,P_3\}$-factor, such that  $G\Box P_i$ is not hamiltonian for every odd $i\leq n$. 
\end{theorem}

%
%
%
%
%

In the rest of the introduction we give the notation and the terminology.
Let $P_n$ be the path on $n$ vertices. 
 The length of a path is the number of edges it contains. A path with one vertex is called a {\em trivial} path.  A set of pairwise disjoint paths $\mathcal P$ 
in a graph $G$ is called a {\em path cover} of $G$ if   $\bigcup_ {P\in \mathcal P} V(P)=V(G)$  (we allow that $\mathcal P$ contains trivial paths). 
Hence the union of paths in  $\mathcal P$  is a {\em path factor} of $G$ -- a spanning subgraph whose components are paths. 
We use the notation $\mathcal P$ to denote a path cover or a path factor, and it is clear from the context which of the two we are referring to. 

 If  $\mathcal P$ is a path cover (or a path factor) of a graph $G$, then we denote by  $\mathcal E_0(\mathcal P)$ the set of endvertices of trivial paths in  $\mathcal P$ and by 
 $\mathcal E_1(\mathcal P)$ the set of endvertices of nontrivial  paths in  $\mathcal P$.  
We define  $\mathcal E(\mathcal P)=\mathcal E_0(\mathcal P)\cup \mathcal E_1(\mathcal P)$. 

 If  $\mathcal P$ is a path cover of a graph $G$, and $R$ is a subgraph of $G$, then let 
 $\mathcal P_R=\{P\cap R\,|\,P\in \mathcal P\}$.  \\

Let $G=(V(G),E(G))$ and $H=(V(H),E(H))$ be graphs. The {\em Cartesian product} of graphs $G$ and $H$ is the graph, denoted as 
$G\Box H$, with vertex set $V(G\Box H)=V(G)\times V(H)$ where 
vertices $(g_1,h_1)$ and   $(g_2,h_2)$ are adjacent in $G\Box H$ if $g_1=g_2$ and $h_1h_2\in E(H)$ or 
$g_1g_2\in E(G)$ and $h_1=h_2$.  \\

Let $G$ be a graph and $H$ a subgraph of $G$. If $x$ and $y$ are vertices of $G$ that are adjacent in $H$ we say that $x$ is an $H$-neighbor of $y$. 

For any positive integer $n$ we define $[n]=\{1,\ldots,n\}$.  We assume (throughout the paper) that $V(P_n)=[n]$ and hence $V(P_3\Box P_n)=[3]\times [n]$.
Let $G=P_3\Box P_n$, and let $E$ and $O$ be subsets of $V(G)$ defined as follows  
$$E=\left\{(i,k)\,|\, i\in [3],  k\  \text{is even}\right\}~{\rm and }~ O=\left\{(i,k)\,|\, i\in [3],\ k\ \text{is odd} \right\}.$$ 

\label{prva}


\section{Proof of the main theorem}
\label{poglavje2}

In this section we prove Theorem \ref{main}.

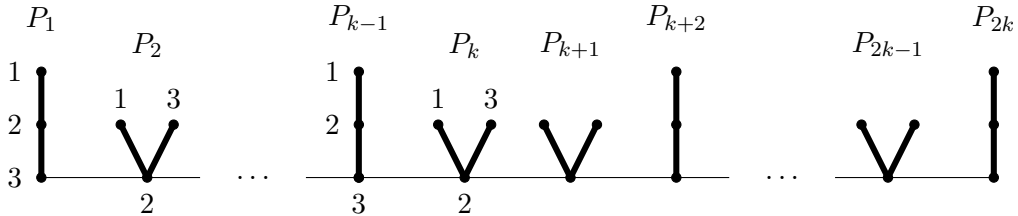
\begin{figure}[h]
\begin{center}
\begin{tikzpicture} [scale=0.7]
\draw (-11,1)--(-8,1);   \path node at (-7,1) { $\ldots$ };
\draw (-6,1)--(2,1);  \path node at (3,1) { $\ldots$ };
\draw (4,1)--(7,1);	

\filldraw (-11,1) circle (2.5pt);\filldraw (-9,1) circle (2.5pt);\filldraw (-5,1) circle (2.5pt); ;\filldraw (-3,1) circle (2.5pt);
;\filldraw (-1,1) circle (2.5pt);
\filldraw (1,1) circle (2.5pt); \filldraw (5,1) circle (2.5pt);\filldraw (7,1) circle (2.5pt);

\draw  [line width=0.8mm] (-3,1)--(-3.5,2);\draw [line width=0.8mm] (-3,1)--(-2.5,2);	
\filldraw (-3.5,2) circle (2.5pt); \filldraw (-2.5,2) circle (2.5pt);
 \path node at (-3,0.5) { $2$ };	 \path node at (-3.5,2.5) { $1$ };	 \path node at (-2.5,2.5) { $3$ };

\draw [line width=0.8mm] (-1,1)--(-1.5,2); \draw  [line width=0.8mm](-1,1)--(-0.5,2);		
\filldraw (-1.5,2) circle (2.5pt); \filldraw (-0.5,2) circle (2.5pt);

\draw  [line width=0.8mm](1,1)--(1,3); 		
	\filldraw (1,2) circle (2.5pt); \filldraw (1,3) circle (2.5pt);	

\draw  [line width=0.8mm] (-5,1)--(-5,3);	
\filldraw (-5,1) circle (2.5pt); \filldraw (-5,2) circle (2.5pt);\filldraw (-5,3) circle (2.5pt);
 \path node at (-5,0.5) { $3$ };	 \path node at (-5.5,2) { $2$ };	 \path node at (-5.5,3) { $1$ };	

\draw [line width=0.8mm] (5,1)--(4.5,2);	\draw  [line width=0.8mm](5,1)--(5.5,2);	
\filldraw (4.5,2) circle (2.5pt); \filldraw (5.5,2) circle (2.5pt);	

\draw  [line width=0.8mm] (7,1)--(7,3);	
\filldraw (7,2) circle (2.5pt); \filldraw (7,3) circle (2.5pt);

\draw [line width=0.8mm] (-9,1)--(-9.5,2); \draw  [line width=0.8mm](-9,1)--(-8.5,2);		
\filldraw (-9.5,2) circle (2.5pt); \filldraw (-8.5,2) circle (2.5pt);
 \path node at (-9,0.5) { $2$ };	 \path node at (-9.5,2.5) { $1$ };	 \path node at (-8.5,2.5) { $3$ };	

\draw  [line width=0.8mm] (-11,1)--(-11,3);	
\filldraw (-11,2) circle (2.5pt); \filldraw (-11,3) circle (2.5pt);
 \path node at (-11.5,1) { $3$ };	 \path node at (-11.5,2) { $2$ };	 \path node at (-11.5,3) { $1$ };	

 \path node at (-11,4) { $P_1$ };  \path node at (-9,3.5) { $P_2$ };\path node at (-5,4) { $P_{k-1}$ };
 \path node at (-3,3.5) { $P_{k}$ }; \path node at (-1,3.5) { $P_{k+1}$ };\path node at (1,4) { $P_{k+2}$ };\path node at (5,3.5) { $P_{2k-1}$ };\path node at (7,4) { $P_{2k}$ }; 
\end{tikzpicture}
\caption{Graph $G_k$}
\label{ggraf}
\end{center}
\end{figure}

Let $G_k$ be the graph shown in Fig.~\ref{ggraf}. Note that $k$ is even and that $G_k$ is a balanced tree. Moreover, observe that $\mathcal P=\{P_i\,|\,i\in [2k]\}$
is a  $\{P_2,P_3\}$-factor of $G_k$. 
In our proof we will rely only on the left half of the graph $G_k$, i.e. the part covered by paths $P_i$ where $i\in [k]$ (the right part serves only to make the graph balanced --- in fact we could have  replaced the right part with any suitable tree that   makes the graph balanced). 
By a misuse of the notation we assume that $V(P_i)=[3]$ for $i\in [k]$. We also assume that the central vertex of each $P_i$ is denoted by 2   
and  that for every odd $i\leq k$,  the vertex of $P_i$ that is   a leaf of $G_k$ is denoted by 1.

Let $k$ be  a given positive even integer. We will prove that for every odd $n< 4k$ the product $G_k\Box P_n$ is not hamiltonian (see Theorem \ref{main1}). 
For a path cover   $\mathcal P$ of $P_3\Box P_n$ we define  
$$\mathcal V_e= E\cap\mathcal E_1(\mathcal P), \overline{\mathcal V}_e=E\cap\mathcal E_0(\mathcal P), \mathcal V_o=O\cap\mathcal E_1(\mathcal P)~ {\rm and}~
\overline{\mathcal V}_o=O\cap\mathcal E_0(\mathcal P).$$
The proof of Theorem \ref{main1} heavily relies on the following two lemmas.

\begin{lemma}
\label{lemasrednji}
Let $n$ be an odd integer and $\mathcal P$ a path cover of $P_3\Box P_n$ such that $\mathcal {E(P)}\subseteq \{2\}\times [n]$. 
  Then we have
  $|\mathcal V_e|+2|\overline{\mathcal V}_e|=|\mathcal V_o|+2|\overline{\mathcal V}_o|+2$.
 
\end{lemma}

\begin{lemma}
\label{lemarob}

Let $n$ be an odd integer and $\mathcal P$ a path cover of $P_3\Box P_n$ such that  $\mathcal {E(P)}\subseteq \{3\}\times [n]$. 
 Then we have   $|\mathcal V_e|+2|\overline{\mathcal V}_e|+2=|\mathcal V_o|+2|\overline{\mathcal V}_o|$.

\end{lemma}

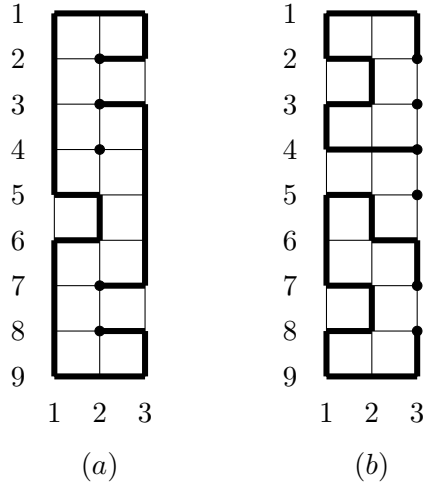
\begin{figure}[htb]\label{z}
\begin{center}
\begin{tikzpicture} [scale=0.6]

\path node at (2,-1) {$(a)$};
\path node at (8,-1) {$(b)$};
\path node at (1,0.2) {$1$};
\path node at (2,0.2) {$2$};
\path node at (3,0.2) {$3$};
\path node at (7,0.2) {$1$};
\path node at (8,0.2) {$2$};
\path node at (9,0.2) {$3$};

\path node at (0.2,1) {$9$};
\path node at (0.2,2) {$8$};
\path node at (0.2,3) {$7$};
\path node at (0.2,4) {$6$};
\path node at (0.2,5) {$5$};
\path node at (0.2,6) {$4$};
\path node at (0.2,7) {$3$};
\path node at (0.2,8) {$2$};
\path node at (0.2,9) {$1$};

\path node at (6.2,1) {$9$};
\path node at (6.2,2) {$8$};
\path node at (6.2,3) {$7$};
\path node at (6.2,4) {$6$};
\path node at (6.2,5) {$5$};
\path node at (6.2,6) {$4$};
\path node at (6.2,7) {$3$};
\path node at (6.2,8) {$2$};
\path node at (6.2,9) {$1$};

\foreach \x in {1,2,3}
\draw  (\x,1)--(\x,9);
\foreach \y in {1,2,3,4,5,6,7,8,9}
\draw (1,\y)--(3,\y);

\foreach \x in {7,8,9}
\draw  (\x,1)--(\x,9);
\foreach \y in {1,2,3,4,5,6,7,8,9}
\draw (7,\y)--(9,\y);

\draw [line width=0.8mm] (1,1)--(1,4);
\draw [line width=0.8mm] (1,5)--(1,9);
\draw [line width=0.8mm] (2,4)--(2,5);
\draw [line width=0.8mm] (3,1)--(3,2);
\draw [line width=0.8mm] (3,3)--(3,7);
\draw [line width=0.8mm] (3,8)--(3,9);
\draw [line width=0.8mm] (1,1)--(3,1);
\draw [line width=0.8mm] (2,2)--(3,2);
\draw [line width=0.8mm] (2,3)--(3,3);
\draw [line width=0.8mm] (1,4)--(2,4);
\draw [line width=0.8mm] (1,5)--(2,5);
\draw [line width=0.8mm] (2,7)--(3,7);
\draw [line width=0.8mm] (2,8)--(3,8);
\draw [line width=0.8mm] (1,9)--(3,9);
\filldraw[fill] (2,2) circle (3pt);
\filldraw[fill] (2,3) circle (3pt);
\filldraw[fill] (2,6) circle (3pt);
\filldraw[fill] (2,7) circle (3pt);
\filldraw[fill] (2,8) circle (3pt);

\draw [line width=0.8mm] (7,1)--(7,2);
\draw [line width=0.8mm] (7,3)--(7,5);
\draw [line width=0.8mm] (7,6)--(7,7);
\draw [line width=0.8mm] (7,8)--(7,9);
\draw [line width=0.8mm] (8,2)--(8,3);
\draw [line width=0.8mm] (8,4)--(8,5);
\draw [line width=0.8mm] (8,7)--(8,8);
\draw [line width=0.8mm] (9,1)--(9,2);
\draw [line width=0.8mm] (9,3)--(9,4);
\draw [line width=0.8mm] (9,8)--(9,9);
\draw [line width=0.8mm] (7,1)--(9,1);
\draw [line width=0.8mm] (7,2)--(8,2);
\draw [line width=0.8mm] (7,3)--(8,3);
\draw [line width=0.8mm] (8,4)--(9,4);
\draw [line width=0.8mm] (7,5)--(8,5);
\draw [line width=0.8mm] (7,6)--(9,6);
\draw [line width=0.8mm] (7,7)--(8,7);
\draw [line width=0.8mm] (7,8)--(8,8);
\draw [line width=0.8mm] (7,9)--(9,9);

\filldraw[fill] (9,2) circle (3pt);
\filldraw[fill] (9,3) circle (3pt);
\filldraw[fill] (9,5) circle (3pt);
\filldraw[fill] (9,6) circle (3pt);
\filldraw[fill] (9,7) circle (3pt);
\filldraw[fill] (9,8) circle (3pt);

\end{tikzpicture}
\caption{A path cover of  $P_3\Box P_9$ where (a) $\mathcal {E(P)}\subseteq \{2\}\times [9]$ and (b) $\mathcal {E(P)}\subseteq \{3\}\times [9]$.}
\label{xxx}
\end{center}
\end{figure}

The proof of the above lemmas is postponed to the following sections. Here we exhibit two examples, shown in Fig.~\ref{z}, of path covers of $P_3\Box P_9$ where in example 
 (a) $\mathcal {E(P)}\subseteq \{2\}\times [9]$ and in (b) $\mathcal {E(P)}\subseteq \{3\}\times [9]$. So for (a) Lemma \ref{lemasrednji} applies and 
for (b) Lemma \ref{lemarob} applies. In case (a) we have  $|\mathcal V_e|=2, |\overline{\mathcal V}_e|=1, |\mathcal V_o|=2$ and $|\overline{\mathcal V}_o|=0$. 
 In case (b) we have  $|\mathcal V_e|=3, |\overline{\mathcal V}_e|=0, |\mathcal V_o|=1$ and $|\overline{\mathcal V}_o|=2$.

The following theorem is a concretization of Theorem \ref{main}. 

\begin{theorem}\label{main1}
For every positive even integer $k$    the graph $G_{k}\Box P_n$ is not hamiltonian   for every odd $n< 4k$. 
\end{theorem}

\proof
Assume the contrary and 
let $n< 4k$ be a given odd integer such that  $G_k\Box P_n$ is hamiltonian.   
Let $C$ be a hamiltonian cycle in $G_k\Box P_n$. 

Define  $\mathcal P_i=C\cap(P_i\Box P_n)$,  
and note that $\mathcal P_i$ is a path factor of $P_i\Box P_n $  for $i\in [k]$. 
Observe that  $\mathcal {E(P}_i)\subseteq \{2\}\times [n]$ for every even $i\in [k]$ and   $\mathcal {E(P}_i)\subseteq \{3\}\times [n]$ for every odd $i\in [k]$. 
Define  $\mathcal V_e^i= E\cap\mathcal E_1(\mathcal P_i)$, $\overline{\mathcal V_e^i}=E\cap\mathcal E_0(\mathcal P_i)$, $\mathcal V_o^i=O\cap\mathcal E_1(\mathcal P_i)$ and
$\overline{\mathcal V_o^i}=O\cap\mathcal E_0(\mathcal P_i)$.

By Lemma \ref{lemasrednji} we have 
\begin{equation}
\label{ena}
|\mathcal V_e^i|+2|\overline{\mathcal V_e^i}|=|\mathcal V_o^i|+2|\overline{\mathcal V_o^i}|+2
\end{equation}  for   even $i\in [k]$, 
and by Lemma \ref{lemarob} we have 
\begin{equation}
\label{dva}
|\mathcal V_e^i|+2|\overline{\mathcal V_e^i}|+2=|\mathcal V_o^i|+2|\overline{\mathcal V_o^i}|
\end{equation} 
for odd $i\in [k]$.

Let $\overrightarrow{\mathcal V_e^i}$ be the set of vertices in $\mathcal V_e^i$ that have a $C$-neighbour in 
$P_{i+1}\Box P_n$, and   let $\overleftarrow{\mathcal V_e^i}$ be the set of vertices in $\mathcal V_e^i$ that have a $C$-neighbour in 
$P_{i-1}\Box P_n$ (analogously we define  $\overrightarrow{\mathcal V_o^i}$ and $\overleftarrow{\mathcal V_o^i}$). 
Note that all vertices in $\overline{\mathcal V_e^i}\cup \overline{\mathcal V_o^i}$ have a $C$-neighbour in $P_{i-1}\Box P_n$ and in $P_{i+1}\Box P_n$. 
Then we have

\begin{equation} \label{ve}
|\mathcal V_e^i|=|\overrightarrow{\mathcal V_e^i}|+|\overleftarrow{\mathcal V_e^i}| 
\ \ \text{and}\ \
|\mathcal V_o^i|=|\overrightarrow{\mathcal V_o^i}|+|\overleftarrow{\mathcal V_o^i}|
\end{equation}  and 
\begin{equation}\label{prehod}
|\overrightarrow{\mathcal V_e^i}|+|\overline{\mathcal V_e^{i}}|=|\overleftarrow{\mathcal V_e^{i+1}}|+|\overline{\mathcal V_e^{i+1}}|
 \ \ \text{and}\ \
|\overrightarrow{\mathcal V_o^i}|+|\overline{\mathcal V_o^{i}}|=|\overleftarrow{\mathcal V_o^{i+1}}|+|\overline{\mathcal V_o^{i+1}}|.
\end{equation}
 
\medskip

{\em Claim:} For every odd $i\in [k]$ we have
 \begin{equation}\label{claimo}
|\overrightarrow{\mathcal V_e^i}|+|\overline{\mathcal V_e^i}|+2i= |\overrightarrow{\mathcal V_o^i}|+|\overline{\mathcal V_o^i}|,\end{equation}
and for every even $i\in [k]$  we have
 \begin{equation}\label{claime}
|\overrightarrow{\mathcal V_e^i}|+|\overline{\mathcal V_e^i}|= |\overrightarrow{\mathcal V_o^i}|+|\overline{\mathcal V_o^i}|+2i.\end{equation}

{\em Proof:} We will prove \eqref{claimo} and \eqref{claime} by induction. 
 By the definition   $|\overleftarrow{\mathcal V_e^1}|=0$, $|\overleftarrow{\mathcal V_o^1}|=0$, 
 $|\overline{\mathcal V_e^1}|=0$ and $|\overline{\mathcal V_o^1}|=0$ and therefore, by \eqref{dva} and \eqref{ve} we have 
$|\overrightarrow{\mathcal V_e^1}|+2=|\overrightarrow{\mathcal V_o^1}|.$
 This  proves  \eqref{claimo} for $i=1$.


We will distinguish two cases: first, $i$ is even, and second,  $i$ is odd.
Assume that claim  
 (\ref{claime}) holds for some even $i<k$. We shall prove that  (\ref{claimo}) holds for $i+1$.  
By   \eqref{dva} we have
$$|\mathcal V_e^{i+1}|+2|\overline{\mathcal V_e^{i+1}}|+2=|\mathcal V_o^{i+1}|+2|\overline{\mathcal V_o^{i+1}}|$$
and therefore, by \eqref{ve},  
$$|\overrightarrow{\mathcal V_e^{i+1}}|+|\overleftarrow{\mathcal V_e^{i+1}}|+2|\overline{\mathcal V_e^{i+1}}|+2=
|\overrightarrow{\mathcal V_o^{i+1}}|+|\overleftarrow{\mathcal V_o^{i+1}}|+2|\overline{\mathcal V_o^{i+1}}|.$$
By (\ref{prehod}) we have 
 $$|\overrightarrow{\mathcal V_e^{i+1}}|+ |\overrightarrow{\mathcal V_e^{i}}|+|\overline{\mathcal V_e^{i}}|
+|\overline{\mathcal V_e^{i+1}}|+2=
|\overrightarrow{\mathcal V_o^{i+1}}|+|\overrightarrow{\mathcal V_o^{i}}|+|\overline{\mathcal V_o^{i}}|+|\overline{\mathcal V_o^{i+1}}|.$$
By  inductive assumption (see (\ref{claime})) we have
$|\overrightarrow{\mathcal V_e^{i}}|+|\overline{\mathcal V_e^{i}}|= |\overrightarrow{\mathcal V_o^{i}}|+|\overline{\mathcal V_o^{i}}|+2i$, so 
$$|\overrightarrow{\mathcal V_e^{i+1}}|+|\overline{\mathcal V_e^{i+1}}|+2i+2=
|\overrightarrow{\mathcal V_o^{i+1}}|+|\overline{\mathcal V_o^{i+1}}|,$$which proves  (\ref{claimo}) for $i+1$. 
Similarly, we prove that if     (\ref{claimo}) holds for an odd $i<k$, then (\ref{claime}) holds for  $i+1$. 
This completes the proof of the claim. 

Now we return to the proof of the theorem. By \eqref{claime} we have (for $i=k$)
$$|\overrightarrow{\mathcal V_e^k}|+|\overline{\mathcal V_e^k}|= |\overrightarrow{\mathcal V_o^k}|+|\overline{\mathcal V_o^k}|+2k.$$
Since $ |\overrightarrow{\mathcal V_o^k}|$ and $|\overline{\mathcal V_o^k}|$ are nonegative, the left side of the above equation is at least $2k$.
Note that $\overrightarrow{\mathcal V_e^k}\cup \overline{\mathcal V_e^k}\subseteq E$ and therefore the number of even $i\leq n$ is at least $2k$. 
It follows that $n\geq 4k$, contradicting the  initial assumption that $n<4k$. \qed 


\section {Proof of Lemma \ref{lemasrednji}}

The aim of this section is to prove Lemma \ref{lemasrednji}.

Throughout this section we   assume that $n$ is  an odd integer, $G=P_3\Box P_n$ and $\mathcal P$ is a path cover of $G$ such that $\mathcal {E(P)}\subseteq \{2\}\times [n]$. 
We will prove that 
  $$|\mathcal V_e|+2|\overline{\mathcal V}_e|=|\mathcal V_o|+2|\overline{\mathcal V}_o|+2.$$ 
We start with  several   preparatory results. For a path cover  $\mathcal P$ of $G$ let $\mathcal F= \bigcup_ {P\in \mathcal P} E(P)$.
In Fig.~\ref{slika1} we show  several possibilities that can occur when $(2,k)\in\mathcal E(\mathcal P)$ for 
some $k\in[n]$.  In each drawing bold lines represent edges that are surely in $\mathcal F$, and 
edges that are denoted by thin lines may  or may not be  in $\mathcal F$.  
Note that the  assumption   $\mathcal {E(P)}\subseteq \{2\}\times [n]$  implies that all vertices of $\{1,3\}\times [n]$ are of degree 2 in  $\mathcal P$. 


\begin{lemma}
\label{moznisredina}
 If $(2,k)\in\mathcal E(\mathcal P)$ then one of the seven cases shown in Fig.~\ref{slika1} occurs.
\end{lemma}

\proof Suppose that $(2,k)\in \mathcal E_0(\mathcal P)$; equivalently $(2,k)$ is an isolated vertex in $\mathcal P$. Since  $\mathcal {E(P)}\subseteq \{2\}\times [n]$, $(1,k)$ and $(3,k)$ are vertices of degree 2  in $\mathcal P$, 
and therefore case $A$ occurs. Now suppose $(2,k)\in \mathcal E_1(\mathcal P)$. If $(2,k)-(3,k)$ is an edge of $\mathcal F$, then since 
$(1,k)$ is a vertex of degree 2 in $\mathcal P$ and   $(1,k)$ is not adjacent to $(2,k)$ in $\mathcal P$, we must have  that
$(1,k-1)-(1,k)-(1,k+1)$ are edges of $\mathcal F$, and hence either case $B$ or case $E$ occurs. 
Similarly, if $(1,k)-(2,k)$ is in  $\mathcal F$, then we have either case $C$ or $F$; if $(2,k)-(2,k+1)$ is in $\mathcal F$, then we have case 
$D$; and if  $(2,k)-(2,k-1)$  is in  $\mathcal F$, then we have case $G$. \qed

\begin{figure}[h]
\begin{center}
\begin{tikzpicture} [scale=0.6]
\draw (1,1) grid (3,3);	
	\draw [line width=0.8mm] (1,1)--(1,3);
	\draw [line width=0.8mm] (3,1)--(3,3);
	\filldraw (2,2) circle (2.5pt);
	 \path node at (2,0) { $\mathcal A_1$ };
\path node at (2,3.5) { $A$ };
	\path node at (0.3,2) {$k$};
	\path node at (0.1,1) {$k+1$};
\path node at (0.1,3) {$k-1$};

  \draw (5,1) grid (7,3);
	\draw [line width=0.8mm] (5,1)--(5,3);
	\draw [line width=0.8mm] (6,2)--(7.05,2);
	\draw [line width=0.8mm] (7,1.95)--(7,3);
		\filldraw (6,2) circle (2.5pt);
		\path node at (6,3.5) { $B$ };

	\draw (8,1) grid (10,3);	
	\draw [line width=0.8mm] (10,1)--(10,3);
	\draw [line width=0.8mm] (7.95,2)--(9,2);
	\draw [line width=0.8mm] (8,1.95)--(8,3);
	\filldraw (9,2) circle (2.5pt);
 \path node at (9,3.5) { $C$ };
	\path node at (9,0) { $\mathcal A_2$ };
		
	\draw (11,1) grid (13,3);	
	\draw [line width=0.8mm] (11,1)--(11,3);
	\draw [line width=0.8mm] (12,1)--(12,2);
	\draw [line width=0.8mm] (13,1)--(13,3);
	\filldraw (12,2) circle (2.5pt);
	\path node at (12,3.5) { $D$ };

	\draw (15,1) grid (17,3);
	\draw [line width=0.8mm] (15,1)--(15,3);
	\draw [line width=0.8mm] (16,2)--(17.05,2);
	\draw [line width=0.8mm] (17,1)--(17,2.05);
	\filldraw (16,2) circle (2.5pt);
	\path node at (16,3.5) { $E$ };

	\draw (18,1) grid (20,3);
	\draw [line width=0.8mm] (20,1)--(20,3);
	\draw [line width=0.8mm] (18,1)--(18,2.05);
	\draw [line width=0.8mm] (18,2)--(19,2);
	\filldraw (19,2) circle (2.5pt);
	\path node at (19,3.5) { $F$ };
\path node at (19,0) { $\mathcal A_3$ };
	
	\draw (21,1) grid (23,3);
	\draw [line width=0.8mm] (21,1)--(21,3);
	\draw [line width=0.8mm] (22,3)--(22,2);
	\draw [line width=0.8mm] (23,1)--(23,3);
	\filldraw (22,2) circle (2.5pt);
	\path node at (22,3.5) { $G$ };
\end{tikzpicture}
\caption{Edges in $\mathcal F$ are denoted by bold lines.}
\label{slika1}
\end{center}
\end{figure}
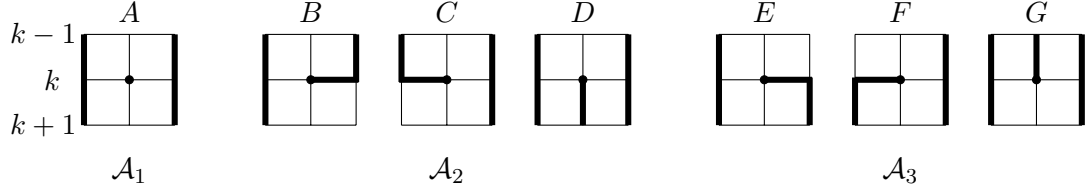

\noindent We define three sets $\mathcal A_1=\{A\}, \mathcal A_2=\{B,C,D\}$ and $\mathcal A_3=\{E,F,G\}$. 
If  $(2,k)\in\mathcal E(\mathcal P)$, we say that (for example)  \textit{{$A$ occurs for $k$}}, if  $\mathcal F$ contains edges denoted by bold lines as    in 
 case $A$  of Fig.~\ref{slika1}. Similarly, we say (for example) that  
 {\textit{$\mathcal A_2$ occurs for $k$}}, if $B,C$ or $D$ occurs for $k$.

\begin{lemma}
\label{lemaliharazdalja}
If $(1,k)-(1,k+1)\in \mathcal F$, $(2,k)-(2,k+1)\notin \mathcal F$,  $(3,k)-(3,k+1)\in \mathcal F$ and  
$([3]\times \{k+1,k+2\})\cap  \mathcal {E(P)}= \emptyset$, then   
 $(1,k+2)-(1,k+3)\in \mathcal F$, $(2,k+2)-(2,k+3)\notin \mathcal F$ and $(3,k+2)-(3,k+3)\in \mathcal F$.
\end{lemma}

\proof
We shall prove that  one of the three cases shown in 
 Fig.~\ref{slikaM1naprej} occurs.  Suppose that  $(1,k+1)-(2,k+1)-(3,k+1)\in \mathcal F$. By assumption   $(1,k+2)$ and $(3,k+2)$ are not contained in $ \mathcal {E(P)}$ and so they are both of degree 2 in some path of $\mathcal P$, which forces that 
$(1,k+3)-(1,k+2)-(2,k+2)-(3,k+2)-(3,k+3)\in \mathcal F$. Hence, we have the first (leftmost) case of  Fig.~\ref{slikaM1naprej}. Similarely, if   $(1,k+1)-(2,k+1)\in  \mathcal F$ and $(2,k+1)-(3,k+1)\notin \mathcal F$, the 
second case shown in Fig.~\ref{slikaM1naprej} is forced (we can see this by  using  the fact that $(1,k+2)\notin \mathcal{E}(\mathcal P)$). Finally, if 
$(1,k+1)-(2,k+1)\notin  \mathcal F$ and $(2,k+1)-(3,k+1)\in \mathcal F$
the third case of  Fig.~\ref{slikaM1naprej} is forced, which  proves the claim. 
\qed

\begin{figure}[h]
\begin{center}
\begin{tikzpicture} [scale=0.8]
{

	\draw (-3,1)--(-1,1); 	\draw (-3,2)--(-1,2);	\draw (-3,3)--(-1,3); \draw (-3,4)--(-1,4);	\draw (-3,1)--(-3,4);	\draw (-2,1)--(-2,4); \draw (-1,1)--(-1,4);
	\draw [line width=0.8mm] (-3,1)--(-3,2.05);
		\draw [line width=0.8mm] (-3,2)--(-0.95,2);
	\draw [line width=0.8mm] (-3,3)--(-0.95,3);
	          \draw [line width=0.8mm] (-3,3-0.05)--(-3,4);
		\draw [line width=0.8mm] (-1,2.95)--(-1,4);
                    \draw [line width=0.8mm] (-1,1)--(-1,2);
	\path node at (-4+0.6,4) {$k$}; 	\path node at (-4+0.3,2) {$k+2$};

	\draw (1,1)--(3,1); 	\draw (1,2)--(3,2);	\draw (1,3)--(3,3); \draw (1,4)--(3,4);	\draw (1,1)--(1,4);	\draw (2,1)--(2,4); \draw (3,1)--(3,4);
	\draw [line width=0.8mm] (1,1)--(1,2.05);
		\draw [line width=0.8mm] (1,2)--(2.05,2);
	\draw [line width=0.8mm] (2,2)--(2,3.05);
\draw [line width=0.8mm] (2,3)--(1,3);
	\draw [line width=0.8mm] (1,3-0.05)--(1,4);
		\draw [line width=0.8mm] (3,1)--(3,4);
	\path node at (0.65,4) {$k$}; 	\path node at (0.3,2) {$k+2$};

		\draw (5,1)--(7,1); 	\draw (5,2)--(7,2);	\draw (5,3)--(7,3); \draw (5,4)--(7,4);	\draw (5,1)--(5,4);	\draw (6,1)--(6,4); \draw (7,1)--(7,4);
	\draw [line width=0.8mm] (5,1)--(5,4);
		\draw [line width=0.8mm] (7,2)--(5.95,2);
	\draw [line width=0.8mm] (6,2)--(6,3.05);
\draw [line width=0.8mm] (6,3)--(7,3);
		\draw [line width=0.8mm] (7,1)--(7,2.05);
			\draw [line width=0.8mm] (7,3-0.05)--(7,4);
	\path node at (4.65,4) {$k$}; \path node at (4.3,2) {$k+2$};
}
\end{tikzpicture}
\caption{The three possibilites that  occur in Lemma \ref{lemaliharazdalja}. }
\label{slikaM1naprej}
\end{center}
\end{figure}
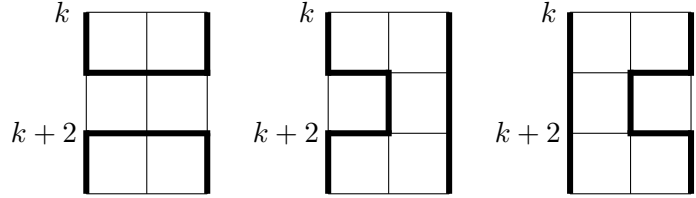

\begin{lemma}
\label{claim2} Suppose that  $(2,k),(2,k')\in\mathcal E(\mathcal P)$ and $(2,t)\notin\mathcal E(\mathcal P)$ for $k<t<k'$. 

\begin{itemize}
\item[(a)] If $\mathcal A_2$ occurs for $k$, then $\mathcal A_1$ or $\mathcal A_2$ does not occur for  $k'$.
\item[(b)] If $\mathcal A_1$ or $\mathcal A_3$ occurs for $k$, then $\mathcal A_3$ does not occur for  $k'$.
\item[(c)] If $\mathcal A_1$ or $\mathcal A_3$ occurs for $k$, and $\mathcal A_1$ or $\mathcal A_2$  occurs for  $k'$, then $k'-k$ is odd.
\end{itemize}
 \end{lemma}

\proof 
(a) Suppose (reductio ad absurdum) that $B,C$ or $D$ occurs for $k$, and  $A, B, C$ or $D$ occurs for $k'$.
Let $R=P_3\Box P_{k+1}^{k'-1}$.  Since $(2,t)\notin\mathcal E(\mathcal P)$ for $k<t<k'$, we have  
$\mathcal E (\mathcal P_ R)\subseteq [3]\times \{k+1,k'-1\}$. Since $B, C$ or $D$ occurs for $k$, we have 
$\mathcal E (\mathcal P_ R)$ has an odd number of vertices in $[3]\times \{k+1\}$ (in cases $B$ or $C$ one vertex, and in case $D$ three vertices), 
and since $A,B,C$ or $D$ occurs for $k'$, 
$\mathcal E (\mathcal P_ R)$ has exactly two vertices in  $[3]\times \{k'-1\}$. Hence $|\mathcal E (\mathcal P_ R)|$ is odd, 
a contradiction (each path of $ \mathcal P_ R$ has exactly two vertices in  $\mathcal E (\mathcal P_ R)$). 

(b) The proof is similar to the proof of (a).

(c) Since $A,E,F$ or $G$ occurs for $k$, we have $(1,k)-(1,k+1)\in \mathcal F$, $(2,k)-(2,k+1)\notin \mathcal F$ and $(3,k)-(3,k+1)\in \mathcal F$. 
 Note that if $k'\neq k+1$, then $k'\neq k+2$. By Lemma \ref{lemaliharazdalja}, $(1,k+2)-(1,k+3)\in \mathcal F$, $(2,k+2)-(2,k+3)\notin \mathcal F$ and $(3,k+2)-(3,k+3)\in \mathcal F$,  and now we use Lemma \ref{lemaliharazdalja}
$\ell$ times (so $\ell$ times one of the three cases shown in  Fig.~\ref{slikaM1naprej}  occurs) until eventually $(2,k+2\ell+1)\in \mathcal {E(P)}$, and so $k+2\ell+1=k'$ and hence  $k'-k$ is odd. 
\qed 

Note that, roughly speaking, in all cases of  $\mathcal A_2$ we have an odd number of bold edges going downwards, and 
in all cases of $\mathcal A_1\cup \mathcal A_2$ we have an even number of bold edges going upwards; this discrepancy explains the above lemma.\\


\begin{lemma} \label{claim4} Suppose that  $(2,k),(2,k')\in\mathcal E(\mathcal P)$ and $(2,t)\notin\mathcal E(\mathcal P)$ for $t<k$ or $t>k'$, where $1\leq k \leq k'\leq n$. 
Then $\mathcal A_3$ does not occur for $k$, and  $\mathcal A_2$ does not occur for  $k'$. Moreover, $k$ and $k'$ are even.\end{lemma}
\proof
The proof  of the first part is similar to the proof  of Lemma \ref{claim2}(a). The proof of the second part  is given in \cite{spa} (see Lemma 2.1 in  \cite{spa}).
\qed


 We finish this section with the proof of Lemma \ref{lemasrednji}. \\

\noindent
 {\bf Proof of Lemma  \ref{lemasrednji}.  }
Let us assume that  $(2,k),(2,k')\in\mathcal E(\mathcal P)$ and $(2,t)\notin\mathcal E(\mathcal P)$ for $k<t<k'$. 
We combine everything given in lemmas above (such as, for example, that case $D$ cannot occur  for $k'$, if $B$ occurs for $k$; in other words, $B$ and $D$ cannot occur one after another) and represent the claims of these lemmas  by a digraph, which we will denote by $T$. 
 We use the notation $\mathcal A_2(e)$ to denote that $B,C$ or $D$ occurs on an even position; all other notations are 
interpreted similarly.


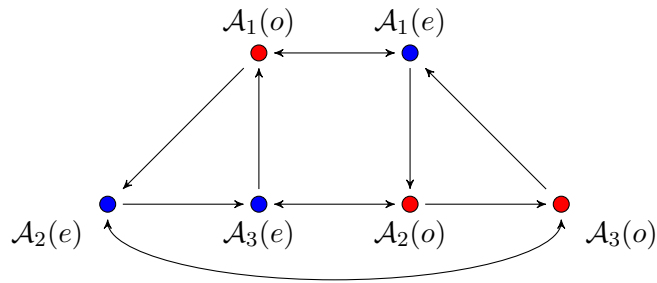
\begin{figure}[h]
\begin{center}
\begin{tikzpicture} 
[->, >=stealth' ];
\filldraw[fill=blue] (1,1) circle (3pt);
\filldraw[fill=blue] (3,1) circle (3pt);
\filldraw[fill=red] (5,1) circle (3pt);
\filldraw[fill=red] (7,1) circle (3pt);
\filldraw[fill=red] (3,3) circle (3pt);
\filldraw[fill=blue] (5,3) circle (3pt);

\path node at (0.2,0.6) { $\mathcal A_2(e)$ };\path node at (3,0.6) { $\mathcal A_3(e)$ };\path node at (5,0.6) { $\mathcal A_2(o)$ };\path node at (7.8,0.6) { $\mathcal A_3(o)$ };\path node at (3,3.4) { $\mathcal A_1(o)$ };
\path node at (5,3.4) { $\mathcal A_1(e)$ };

\draw [->] (1.2,1) -- (2.8,1);
\draw [<->] (3.2,1) -- (4.8,1); 
\draw [->] (5.2,1) -- (6.8,1);
\draw [<-] (1.2,1.2) -- (2.8,2.8);
\draw [<->] (3.2,3) -- (4.8,3); 
\draw [->] (3,1.2) -- (3,2.8);
\draw [<-] (5,1.2) -- (5,2.8);
\draw [->] (6.8,1.2) -- (5.2,2.8);
\draw [<->](1,0.8) arc [start angle=180,
                                  end angle=360,
                                 x radius=3,
                                 y radius=0.8];
 
\end{tikzpicture}
\caption{The graph T.}
\label{slika3}
\end{center}

\end{figure}

First consider the vertex $\mathcal A_2(e)$. If $\mathcal A_2$ occurs for an even $k$, 
then by Lemma \ref{claim2}(a), $\mathcal A_1$ or $\mathcal A_2$ does not occur for $k'$. This explains why
there is no arc from $\mathcal A_2(e)$ to $\mathcal A_1(e), \mathcal \mathcal A_1(o), \mathcal A_2(e)$ or $\mathcal A_2(o)$. So the only possibility is that $\mathcal A_3$ occurs for $k'$; this explains the arcs 
$\mathcal A_2(e)\rightarrow \mathcal A_3(e)$ and $\mathcal A_2(e)\rightarrow \mathcal A_3(o)$. 

Now consider the vertex $\mathcal A_3(e)$ (and assume that $\mathcal A_3$ occurs for $k$, and that $k$ is even). By Lemma \ref{claim2}(b), $\mathcal A_3$ does not appear for $k'$, and by Lemma \ref{claim2}(c), $k'$ is odd. 
This explains why there are no arcs from $\mathcal A_3(e)$ to $\mathcal A_3(e),\mathcal A_3(o), \mathcal A_1(e)$ or $\mathcal A_2(e)$. 
The situations $\mathcal A_3(e)\rightarrow \mathcal A_1(o)$ and $\mathcal A_3(e)\rightarrow \mathcal A_2(o)$ are possible (so we have these two arcs in $T$).

Consider the vertex $\mathcal A_2(o)$ (and assume that $\mathcal A_2$ occurs for $k$, and that $k$ is odd). By Lemma \ref {claim2}(a), $\mathcal A_1$ or 
$\mathcal A_2$ does not appear for $k'$. This explains why there are no arcs from $\mathcal A_2(o)$ to 
$\mathcal A_1(e), \mathcal A_1(o), \mathcal A_2(e)$ or $\mathcal A_2(o)$. 

Consider the vertex $\mathcal A_3(o)$ (and assume that $\mathcal A_3$ occurs for $k$, and that $k$ is odd). By Lemma \ref{claim2}(b), $\mathcal A_3$ does not appear for $k'$, and by Lemma \ref{claim2}(c), $k'$ is even. This explains why there are no arcs from $\mathcal A_3(o)$ to 
$\mathcal A_3(e)$, $\mathcal A_3(o), \mathcal A_1(o)$ or $\mathcal A_2(o)$. 

 Consider the vertex $\mathcal A_1(o)$ (and assume that $\mathcal A_1$ occurs for $k$, and that $k$ is odd). By Lemma \ref{claim2}(b), $\mathcal A_3$ does not appear for $k'$, and by Lemma \ref{claim2}(c), $k'$ is even. 
This explains why there are no arcs from $\mathcal A_1(o)$ to $\mathcal A_3(e)$, $\mathcal A_3(o)$, $\mathcal A_1(o)$ or  $\mathcal A_2(o)$. 

Finally, consider the vertex $\mathcal A_1(e)$ (and assume that $\mathcal A_1$ occurs for $k$, and that $k$ is even). By Lemma \ref{claim2}(b), $\mathcal A_3$ does not appear for $k'$, and by Lemma \ref{claim2}(c), $k'$ is odd.  This explains why there are no arcs from $\mathcal A_1(e)$ to 
$\mathcal A_3(e), \mathcal A_3(o), \mathcal A_1(e)$ or  $\mathcal A_2(e)$. 

This gives a complete description of $T$. Observe that for a given path cover $\mathcal P$, each vertex of  $\mathcal E(\mathcal P)$ 
corresponds to a vertex of $T$ (since, by Lemma \ref{moznisredina}, for each  vertex of  $\mathcal E(\mathcal P)$  one of the cases  $\mathcal A_1(e),
 \mathcal A_1(o), \mathcal A_2(e), \mathcal A_2(o),\mathcal A_3(e)$ or  $\mathcal A_3(o)$ occurs). Moreover, the sequence 
$(2,k_1),\ldots,(2,k_r)$, where $k_1<\ldots<k_r$, of vertices in $\mathcal E(\mathcal P)$, corresponds to a directed path in $T$. 
We denote the vertices of this path by $x_1,\ldots,x_r$, and in the rest of this section, we analyse its  properties.

The vertices $x_1$ and $x_r\in V(T)$ correspond to $(2,k_1)$ and $(2,k_r)$, respectively.  
By Lemma \ref{claim4}  we have 
 $x_1\in\{\mathcal A_1(e), \mathcal A_2(e)\}:=\mathcal S$ and 
 $x_r\in \{\mathcal A_1(e), \mathcal A_3(e)\}:=\mathcal Q$ (here we consider $\mathcal A_1(e),
 \mathcal A_1(o), \mathcal A_2(e), \mathcal A_2(o),\mathcal A_3(e)$ and  $\mathcal A_3(o)$ as vertices of $T$). 

Let $I=\{i\, |\, x_i\in \mathcal S\}$ and $i_0=\max\{i\, |\, x_i\in \mathcal S\}$.
Suppose that $i$ and $j$ are integers such that $x_i, x_j\in \mathcal S$ and $x_\ell\notin \mathcal S$ for $i<\ell<j$. 
 Define $X_i=\{x_i,\ldots, x_{j-1}\}$ for $i\in I\setminus \{i_0\}$ and $X_{i_0}=\{x_{i_0},\ldots, x_{r}\}$.
 For $i\in I$ let

\begin{quote} 
$a_i=|X_i\cap \{\mathcal A_2(e), \mathcal A_3(e)\}|$, 
$ b_i=|X_i\cap \{\mathcal A_2(o), \mathcal A_3(o)\}|$,
$c_i=|X_i\cap \{\mathcal A_1(e)\}|$, 
$d_i=|X_i\cap \{\mathcal A_1(o)\}|$.
\end{quote}
Note that, by the definition given in Section \ref{poglavje2} (see also  Fig.~\ref{slika1}), we have  
\begin{equation}
\label{def}
|\mathcal V_e|=\sum_{i\in I}a_i, |\overline{\mathcal V_e}|=\sum_{i\in I}c_i,\\  |\mathcal V_o|=\sum_{i\in I}b_i\; \text{and}\; |\overline{\mathcal V_o}|=\sum_{i\in I}d_i.  
\end{equation} 

Recall that $x_{i}, x_j\in \mathcal S=\{\mathcal A_1(e),\mathcal A_2(e)\}$ and  $x_r\in \mathcal Q=\{\mathcal A_1(e), \mathcal A_3(e)\}$. We claim that $a_i+2c_i=b_i+2d_i$ for $i\in I\setminus \{{i_0}\}$  and $a_{i_0}+2c_{i_0}=b_{i_0}+2d_{i_0}+2$.   Suppose that   $x_i=\mathcal A_2(e)$. Then $x_{i+1}=\mathcal A_3(o)$ or $x_{i+1}=\mathcal A_3(e)$ (see graph T in Fig.~\ref{slika3}) and therefore either case  (1) or case (2) written below occurs. 
If $x_i=\mathcal A_1(e)$, then  case (3)  or case (4) occurs. 


\begin{itemize}
\item[(1)] $\mathcal A_2(e) \rightarrow  \mathcal A_3(o) \rightarrow  \mathcal S.$

\item[(2)] $\mathcal A_2(e) \rightarrow  \mathcal A_3(e) \rightarrow\underbrace{\mathcal A_2(o) \rightarrow \mathcal A_3(e)\rightarrow \ldots \rightarrow  \mathcal A_2(o) \rightarrow \mathcal A_3(e)}_{\text{some number (possibly 0) of repetitions of } \;  \mathcal A_2(o) \rightarrow \mathcal A_3(e)}$ and then one of the following possibilities occurs:

\begin{itemize}
\item[(2.1)] $\mathcal A_1(o)\rightarrow \mathcal S.$ 
\item[(2.2)]  $\mathcal A_2(o) \rightarrow  \mathcal A_3(o)\rightarrow \mathcal S.$ 
\end{itemize}

\item[(3)] $\mathcal A_1(e) \rightarrow  \mathcal A_1(o) \rightarrow \mathcal S.$
\item[(4)] $\mathcal A_1(e) \rightarrow  \underbrace{\mathcal A_2(o) \rightarrow \mathcal A_3(e)\rightarrow \ldots \rightarrow  \mathcal A_2(o) \rightarrow \mathcal A_3(e)}_{\text{some number (possibly 0) of repetitions of } \;  \mathcal A_2(o) \rightarrow \mathcal A_3(e)}$ and then we have the subcases (2.1) or (2.2).

\end{itemize}
Now we return to the proof of the claim, that $a_i+2c_i=b_i+2d_i$ for $i\in I\setminus \{{i_0}\}$. In case (1),  $\mathcal A_2(e)$ contributes 1 to $a_i+2c_i$ and $\mathcal A_3(o)$ contributes 1 to $b_i+2d_i$.  Hence, $a_i+2c_i=b_i+2d_i$ is true in this case.  In case (2), $\mathcal A_2(e)$ and  $\mathcal A_3(e)$ together contribute 2 to $a_i+2c_i$. Note that each repetition of  $\mathcal A_2(o) \rightarrow \mathcal A_3(e)$ contributes 1 to both  $a_i+2c_i$ and  $b_i+2d_i$. After  reaching the last occurrence of $\mathcal A_3(e)$,
the contribution of vertices in $X_i$ to $a_i+2c_i$ is two more than the contribution to $b_i+2d_i$. In subcase (2.1), $\mathcal A_1(o)$ contributes 2 to $b_i+2d_i$, and in subcase (2.2), $\mathcal A_2(o) \rightarrow  \mathcal A_3(o)$ contributes 2 to $b_i+2d_i$. Therefore, the claim $a_i+2c_i=b_i+2d_i$ holds in case (2). 

This claim also holds in case (3), since $\mathcal A_1(e) \rightarrow  \mathcal A_1(o)$  contributes  2 to both  $a_i+2c_i$ and $b_i+2d_i$. In case (4),  $\mathcal A_1(e)$ contributes 2 to $a_i+2c_i$ and  each repetition of  $\mathcal A_2(o) \rightarrow \mathcal A_3(e)$ contributes equally to $a_i+2c_i$ and $b_i+2d_i$. When we reach  the last occurrence of $\mathcal A_3(e)$, the contribution to $a_i+2c_i$ is two more than the contribution to $b_i+2d_i$. Since the path $X_i$ ends with one of the subcases (2.1) or (2.2), which we have already dealt above, we conclude that 
$a_i+2c_i=b_i+2d_i$.


Now consider the path $\{x_{i_0},\ldots, x_{r}\}$. We claim that $a_{i_0}+2c_{i_0}=b_{i_0}+2d_{i_0}+2$. Recall that $x_r\in \mathcal Q=\{\mathcal A_1(e), \mathcal A_3(e)\}$  and suppose first that $x_r=\mathcal A_1(e)\in \mathcal S$. Note that by the claim proved above vertices $ x_{i_0},\ldots, x_{r-1}$ contribute equally to $a_{i_0}+2c_{i_0}$ and to $b_{i_0}+2d_{i_0}+2$. The vertex $x_r$ contributes 2 to  $a_{i_0}+2c_{i_0}$. Hence,  $a_{i_0}+2c_{i_0}=b_{i_0}+2d_{i_0}+2$. If  $x_r=\mathcal A_3(e)$, the proof is similar (case (2) and case (4) above). 

To finish the proof of Lemma \ref{lemasrednji} we use both claims, namely  that $a_{i_0}+2c_{i_0}=b_{i_0}+2d_{i_0}+2$ and that
$a_{i}+2c_{i}=b_{i}+2d_{i}$ for $i\in I\setminus \{i_0\}$. By  \eqref{def} we have (see also the definition of $\mathcal V_e$ and  $\mathcal   V_o$ in Section  \ref{poglavje2}):
\begin{eqnarray*} 
|\mathcal V_e|+2|\overline{\mathcal V}_e|&=&a_{i_0}+2c_{i_0}+\sum_{i\in I\setminus\{i_0\}}a_i+\sum_{i\in I\setminus\{i_0\}}2c_i\\
&=&b_{i_0}+2d_{i_0}+2+\sum_{i\in I\setminus\{i_0\}}b_i+\sum_{i\in I\setminus\{i_0\}}2d_i\\
&=&\sum_{i\in I}b_i+2\sum_{i\in I}d_i+2\\
&=& |\mathcal V_o|+2|\overline{\mathcal V}_o|+2.
\end{eqnarray*}


\section {Proof of Lemma \ref{lemarob} }

The aim of this section is to prove Lemma \ref{lemarob}.

Throughout this section we   assume that $n$ is  an odd integer, $G=P_3\Box P_n$ and $\mathcal P$ is a path cover of $G$ such that $\mathcal {E(P)}\subseteq \{3\}\times [n]$. 
We will prove that 
  $$|\mathcal V_e|+2|\overline{\mathcal V}_e|+2=|\mathcal V_o|+2|\overline{\mathcal V}_o|.$$ 

%
%
We start with a few    preparatory results.

\begin{lemma} \label{moznih19}
 If $(3,k)\in\mathcal E(\mathcal P)$ then one of the nineteen cases shown in Fig.~\ref{slika4} occurs. 
\end {lemma}

\begin{figure}[h]
\begin{center}
\begin{tikzpicture} [scale=0.6]
	
\draw (1,1) grid (3,3);	
\path node at (2,3.6) { $N$ };
\filldraw (3,2) circle (2.5pt);
\draw [line width=0.8mm] (1,1.95)--(1,3);
\draw [line width=0.8mm] (0.95,2)--(2.05,2);
\draw [line width=0.8mm] (2,2.05)--(2,1);
\draw [line width=0.8mm] (3,1)--(3,2);
\path node at (0.3,2) {$k$};
	\path node at (0.1,1) {$k+1$};
\path node at (0.1,3) {$k-1$};

\draw (4,1) grid (6,3);
\path node at (5,3.6) { $O$ };
\filldraw (6,2) circle (2.5pt);
\draw [line width=0.8mm] (4,1)--(4,3);
\draw [line width=0.8mm] (5,1)--(5,2.05);
\draw [line width=0.8mm] (4.95,2)--(6,2);

\draw (7,1) grid (9,3);
\path node at (8,3.6) { $P$ };
\filldraw (9,2) circle (2.5pt);
\draw [line width=0.8mm] (7,1)--(7,2.05);
\draw [line width=0.8mm] (6.95,2)--(8.05,2);
\draw [line width=0.8mm] (8,1)--(8,2.05);
\draw [line width=0.8mm] (9,2)--(9,3);

\path node at (9.5,0.2) { $\mathcal B_7$ };

\draw (10,1) grid (12,3);
\path node at (11,3.6) { $Q$ };
\filldraw (12,2) circle (2.5pt);
\draw [line width=0.8mm] (10,1)--(10,3);
\draw [line width=0.8mm] (11,1)--(11,3);
\draw [line width=0.8mm] (12,2)--(12,3);

\draw (13,1) grid (15,3);
\path node at (14,3.6) { $R$ };
\filldraw (15,2) circle (2.5pt);
\draw [line width=0.8mm] (12.95,2)--(15,2);
\draw [line width=0.8mm] (13,1.95)--(13,3);

\draw (16,1) grid (18,3);
\path node at (17,3.6) { $S$ };
\filldraw (18,2) circle (2.5pt);
\draw [line width=0.8mm] (16,1.95)--(16,3);
\draw [line width=0.8mm] (15.95,2)--(17.05,2);
\draw [line width=0.8mm] (17,1.95)--(17,3);
\draw [line width=0.8mm] (18,2)--(18,3);


\draw (1,6) grid (3,8);	
\path node at (2,8.6) { $K$ };
\filldraw (3,7) circle (2.5pt);
\draw [line width=0.8mm] (1,6)--(1,8);
\draw [line width=0.8mm] (2,6)--(2,8);
\path node at (0.3,7) {$k$};
	\path node at (0.1,6) {$k+1$};
\path node at (0.1,8) {$k-1$};

\draw (4,6) grid (6,8);
\path node at (5,8.6) { $L$ };
\filldraw (6,7) circle (2.5pt);
\draw [line width=0.8mm] (4,6.95)--(4,8);
\draw [line width=0.8mm] (5,6.95)--(5,8);
\draw [line width=0.8mm] (3.95,7)--(5.05,7);

\path node at (5,5.2) { $\mathcal B_6$ };

\draw (7,6) grid (9,8);
\path node at (8,8.6) { $M$ };
\filldraw (9,7) circle (2.5pt);
\draw [line width=0.8mm] (7,6)--(7,7.05);
\draw [line width=0.8mm] (8,6)--(8,7.05);
\draw [line width=0.8mm] (6.95,7)--(8.05,7);

\draw (1,11) grid (3,13);
\path node at (2,13.6) { $E$ };
\filldraw (3,12) circle (2.5pt);
\draw [line width=0.8mm] (1,11)--(1,12.05);
\draw [line width=0.8mm] (0.95,12)--(2.05,12);
\draw [line width=0.8mm] (2,11.95)--(2,13);
\draw [line width=0.8mm] (3,12)--(3,13);
\path node at (0.3,12) {$k$};
	\path node at (0.1,11) {$k+1$};
\path node at (0.1,13) {$k-1$};
	
\draw (4,11) grid (6,13);
\path node at (5,13.6) { $F$ };
\filldraw (6,12) circle (2.5pt);
\draw [line width=0.8mm] (4,11)--(4,13);
\draw [line width=0.8mm] (5,11.95)--(5,13);
\draw [line width=0.8mm] (4.95,12)--(6,12);

\draw (7,11) grid (9,13);
\path node at (8,13.6) { $G$ };
\filldraw (9,12) circle (2.5pt);
\draw [line width=0.8mm] (7,11)--(7,13);
\draw [line width=0.8mm] (8,11)--(8,13);
\draw [line width=0.8mm] (9,11)--(9,12);

\path node at (9.5,10.2) { $\mathcal B_5$ };

\draw (10,11) grid (12,13);
\path node at (11,13.6) { $H$ };
\filldraw (12,12) circle (2.5pt);
\draw [line width=0.8mm] (10,11)--(10,12.05);
\draw [line width=0.8mm] (9.95,12)--(12,12);

\draw (13,11) grid (15,13);
\path node at (14,13.6) { $I$ };
\filldraw (15,12) circle (2.5pt);
\draw [line width=0.8mm] (13,11.95)--(13,13);
\draw [line width=0.8mm] (14,11.95)--(14,13);
\draw [line width=0.8mm] (12.95,12)--(14.05,12);
\draw [line width=0.8mm] (15,12)--(15,11);

\draw (16,11) grid (18,13);
\path node at (17,13.6) { $J$ };
\filldraw (18,12) circle (2.5pt);
\draw [line width=0.8mm] (16,11)--(16,12.05);
\draw [line width=0.8mm] (17,11)--(17,12.05);
\draw [line width=0.8mm] (18,11)--(18,12);
\draw [line width=0.8mm] (15.95,12)--(17.05,12);

\draw (1,16) grid (3,18);
\path node at (2,15.2) { $\mathcal B_1$ };
\path node at (2,18.6) { $A$ };
\filldraw (3,17) circle (2.5pt);
\draw [line width=0.8mm] (1,16)--(1,17.05); 
\draw [line width=0.8mm] (0.95,17)--(2.05,17); 
\draw [line width=0.8mm] (2,16.95)--(2,18); 
\draw [line width=0.8mm] (3,16)--(3,17); 
\path node at (0.3,17) {$k$};
	\path node at (0.1,16) {$k+1$};
\path node at (0.1,18) {$k-1$};

\draw (4,16) grid (6,18);
\path node at (5,18.6) { $B$ };
\path node at (5,15.2) { $\mathcal B_2$ };
\filldraw (6,17) circle (2.5pt);
\draw [line width=0.8mm] (4,16)--(4,17.05);
\draw [line width=0.8mm] (5,18)--(5,16.95);
\draw [line width=0.8mm] (3.95,17)--(5.05,17);

\draw (7,16) grid (9,18);
\path node at (8,18.6) { $C$ };
\path node at (8,15.2) { $\mathcal B_3$ };
\filldraw (9,17) circle (2.5pt);
\draw [line width=0.8mm] (7,16.95)--(7,18);
\draw [line width=0.8mm] (8,17.05)--(8,16);
\draw [line width=0.8mm] (6.95,17)--(8.05,17);
\draw [line width=0.8mm] (9,17)--(9,18);

\draw (10,16) grid (12,18);
\path node at (11,18.6) { $D$ };
\path node at (11,15.2) { $\mathcal B_4$ };
\filldraw (12,17) circle (2.5pt);
\draw [line width=0.8mm] (10,16.95)--(10,18);
\draw [line width=0.8mm] (11,17.05)--(11,16);
\draw [line width=0.8mm] (9.95,17)--(11.05,17);

\end{tikzpicture}
\caption{Edges in $\mathcal F$ are denoted by bold lines, nineteen cases.}
\label{slika4}
\end{center}

\end{figure}
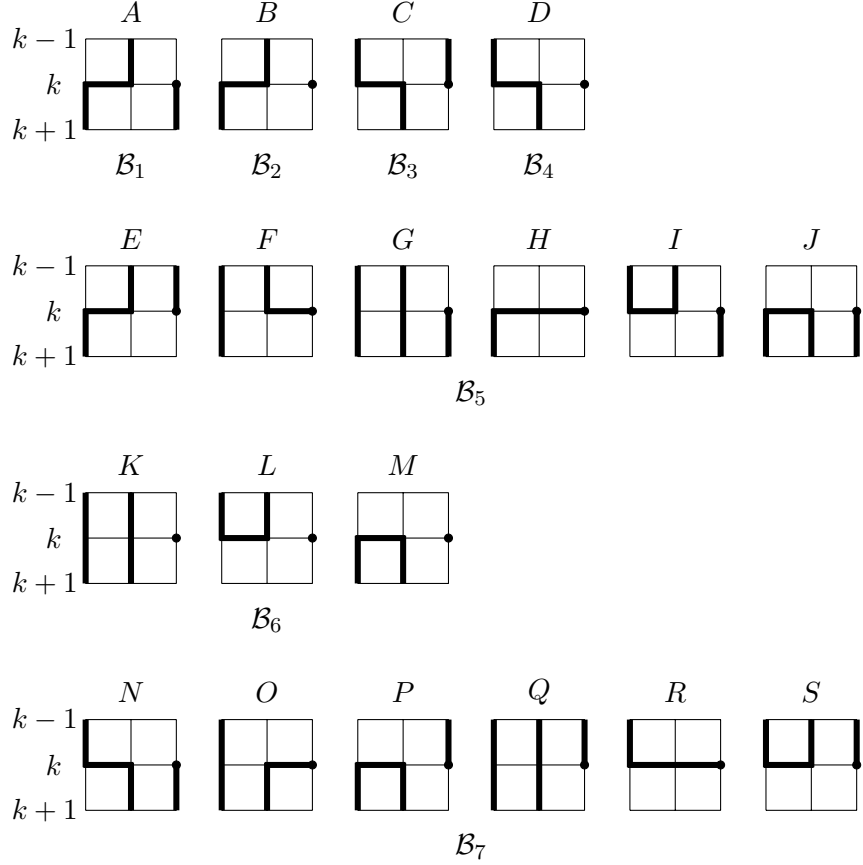

\proof
If $(3,k)\in \mathcal E_1(\mathcal P)$, then we have several possibilities.
If $(3,k)-(3,k+1)$ is an edge of $\mathcal F$, then $(2,k)-(3,k)$ and $(3,k-1)-(3,k)$ are definitely not edges of $\mathcal F$. In the case when $(1,k)-(2,k)$ is an edge of $\mathcal F$, we have one of the cases $A$, $I$, $J$ or $N$;  otherwise, when $(1,k)-(2,k)$ is not an edge of $\mathcal F$,
vertices $(1,k)$ and $(2,k)$ are both of degree $2$ in $G_{\mathcal P}$ and so we have case $G$. 
If $(3,k-1)-(3,k)$ is an edge of $\mathcal F$, symmetrically we get 
the cases $C$, $E$, $P$, $S$ and $Q$.
If $(2,k)-(3,k)$ is an edge of $\mathcal F$, then $(3,k-1)-(3,k)$ and $(3,k)-(3,k+1)$ are not edges of $\mathcal F$. In case when $(1,k)-(2,k)$ is an edge 
of $\mathcal F$, we have  case $R$ or $H$, otherwise we have  case $F$ or $O$. \\
If $(3,k)\in \mathcal E_0(\mathcal P)$, then $(2,k)-(3,k)$, $(3,k)-(3,k-1)$ and $(3,k)-(3,k+1)$ are not edges of $\mathcal F$. In the case when $(1,k)-(2,k)$ is an edge of $\mathcal F$, we have one of the cases $B$, $D$, $L$ or $M$; otherwise, if $(1,k)-(2,k)$ is not an edge of $\mathcal F$, we have case $K$.
 \qed

We define seven sets 

\begin{quote} 
$\mathcal B_1=\{A\}\, , \, \mathcal B_2=\{B\}\, , \, \mathcal B_3=\{C\}\, , \, \mathcal B_4=\{D\}\, , \, \mathcal B_5=\{E,F,G,H,I,J\}$\\
$ \mathcal B_6=\{K,L,M\}$ and $\mathcal B_7=\{N,O,P,Q,R,S\}$
\end{quote}


\begin{lemma} \label{paj}
Suppose that  $(3,k),(3,k')\in\mathcal E(\mathcal P)$ and $(3,t)\notin\mathcal E(\mathcal P)$ for $k<t<k'$. 
\begin{itemize}
\item[(a)] If $\mathcal B_2$ or $\mathcal B_5$ occurs for $k$, then $\mathcal B_1,\mathcal B_2,\mathcal B_3, \mathcal B_5$ or $\mathcal B_6$ does not occur for  $k'$.
\item[(b)] If $\mathcal B_1, \mathcal B_6$ or $\mathcal B_7$ occurs for $k$, then $\mathcal B_1,\mathcal B_2,\mathcal B_4$ or $\mathcal B_7$ does not occur for  $k'$.   
\item[(c)] If  $\mathcal B_1$ occurs for $k$, then either $\mathcal B_3$ occurs for $k'$ and $k'-k$ is odd, or 
$\mathcal B_5$ or $\mathcal B_6$ occurs for  $k'$ and $k'-k$ is even.  
 \item[(d)] Suppose that $\mathcal B_6$ or $\mathcal B_7$ occurs for $k$. If $\mathcal B_5$ or $\mathcal B_6$ occurs for  $k'$, then $k'-k$ is odd, and if $\mathcal B_3$ occurs for  $k'$, then $k'-k$ is even.   
 \item[(e)] If $\mathcal B_3$ or $\mathcal B_4$ occurs for $k$, then $\mathcal B_1$ or $\mathcal B_2$ occurs for  $k'=k+1$. 
\end{itemize}
\end {lemma}

\proof 
\begin{itemize}
\item[(a)]  If $\mathcal B_2$ or $\mathcal B_5$ occurs for $k$ and $\mathcal B_3, \mathcal B_5$ or $\mathcal B_6$ occurs for $k'$, then as in the proof of Lemma \ref{claim2}, we find  that $|\mathcal E (\mathcal P_ R)|$ is odd, where 
$R=P_3\Box P_{k+1}^{k'-1}$ (see Lemma \ref{claim2} for details), which is a contradiction. Otherwise, if   $\mathcal B_1$ or $\mathcal B_2$ occurs for  $k'$, then it is easy to see that either  $(1,k'-1)$
or $(3,k'-1)$ is not covered. 

\item[(b)]  The proof is similar to the proof of Lemma \ref{claim2}. 

\item[(c)] Suppose that $\mathcal B_1$ occurs for $k$. Note that by (b), if $\mathcal B_1$ occurs for $k$, then $\mathcal B_3, \mathcal B_5$ or $\mathcal B_6$ occurs for $k'$. Let  $\mathcal B_3$ occurs for $k'$. We have $(1,k)-(1,k+1)\in \mathcal F$, $(2,k)-(2,k+1)\notin \mathcal F$ and $(3,k)-(3,k+1)\in \mathcal F$
Since  $(3,t)\notin\mathcal E(\mathcal P)$ for $k<t<k'$, we can apply Lemma \ref{lemaliharazdalja} multiple times, say $\ell$ times, until $\mathcal B_3$ occurs for $k'$, and in this case  $k'=k+2\ell+1$. Therefore $k'-k$ is odd. If $\mathcal B_5$ or $\mathcal B_6$ occurs for  $k'$, then the proof is similar (we use Lemma \ref{lemaliharazdalja}). 

\item[(d)]  The method of proving case (d) is similar to the proof of case (c) (we use Lemma \ref{lemaliharazdalja}). 

\item[(e)] Since $\mathcal B_3$ or $\mathcal B_4$ occurs for $k$, we have $(1,k)-(1,k+1)\notin \mathcal F$, $(2,k)-(2,k+1)\in \mathcal F$ and $(3,k)-(3,k+1)\notin \mathcal F$.
By assumption   $(1,k+1)$ is not contained in $\mathcal {E(P)}$ and so it is of degree 2 in some path of $\mathcal P$, which forces that 
$(1,k+2)-(1,k+1)-(2,k+1)\in \mathcal F$ and hence $(2,k+1)-(3,k+1)\notin \mathcal F$. So we have $(3,k+1)\in \mathcal {E(P)}$. 
If $(3,k+1)-(3,k+2)\in \mathcal F$, then case $\mathcal B_1$ occurs for $k+1$;  otherwise  $\mathcal B_2$ occurs for $k+1$.
\qed 

\end{itemize}

\begin{lemma}\label{zacetnizadnji}
 Suppose that  $(3,k),(3,k')\in\mathcal E(\mathcal P)$ and $(3,t)\notin\mathcal E(\mathcal P)$ for $t<k$ or $t>k'$, where $1\leq k \leq k'\leq n$. 
Then $\mathcal B_1, \mathcal B_2,\mathcal B_4$ or $\mathcal B_7$ does not occur for $k$, and  $\mathcal B_2,\mathcal B_3,\mathcal B_4$  or $\mathcal B_5$ does not occur for  $k'$. Moreover, if $\mathcal B_3$ occurs for $k$, then $k$ is even and if $\mathcal B_5$ or $\mathcal B_6$ occurs for $k$, then $k$ is odd.  
If $\mathcal B_1$ occurs for $k'$, then $k'$ is even and if $\mathcal B_6$ or $\mathcal B_7$ occurs for $k'$, then $k'$ is odd.  
\end{lemma}
\proof
The proof  of the first part is similar to the proof of Lemma \ref{paj}(a). 
The proof of the moreover part of the lemma is similar to the proof of 
Lemma 2.1 in  \cite{spa}. \qed

 \begin{figure}[h]
\begin{center}
\begin{tikzpicture} [scale=0.95]

\filldraw[fill=red] (1,1) circle (3pt);
\filldraw[fill=blue] (1,3) circle (3pt);
\filldraw (1,5) circle (3pt);
\filldraw[fill=red] (1,7) circle (3pt);
\filldraw[fill=blue] (3,1) circle (3pt);
\filldraw[fill=red] (3,3) circle (3pt);
\filldraw[fill=red] (3,5) circle (3pt);
\filldraw[fill=red] (5,1) circle (3pt);
\filldraw[fill=blue] (5,5) circle (3pt);
\filldraw[fill=blue] (7,1) circle (3pt);
\filldraw[fill=blue] (7,3) circle (3pt);
\filldraw[fill=red] (7,5) circle (3pt);
\filldraw[fill=blue] (9,1) circle (3pt);
\filldraw (9,3) circle (3pt);
\filldraw[fill=red] (9,5) circle (3pt);
\filldraw[fill=blue] (9,7) circle (3pt);

\path node at (1,0.5) { $\mathcal B_3(e)$ }; \path node at (3,0.5) { $\mathcal B_2(o)$ };
\path node at (5,0.5) { $\mathcal B_4(e)$ }; \path node at (7,0.5) { $\mathcal B_5(o)$ };
\path node at (9,0.5) { $\mathcal B_6(o)$ };
\path node at (0.3,3) { $\mathcal B_1(o)$ }; \path node at (0.3,5) { $\mathcal{S}$ };
\path node at (0.3,7) { $\mathcal B_6(e)$ };
\path node at (9.5,3) { $\mathcal{R}$ }; \path node at (9.8,5) { $\mathcal B_1(e)$ };
\path node at (9.8,7) { $\mathcal B_3(o)$ };
\path node at (3,5.5) { $\mathcal B_5(e)$ }; \path node at (5,5.5) { $\mathcal B_7(o)$ };
\path node at (6.7,5.5) { $\mathcal B_2(e)$ };
\path node at (2.4,2.9) { $\mathcal B_7(e)$ };
\path node at (7.8,2.9) { $\mathcal B_4(o)$ };

\draw [->] (1,1.2) -- (1,2.8); 
\draw [->] (1,3.2) -- (1,4.8);
\draw [<-] (1,5.2) -- (1,6.8);
\draw [->] (3,1.2) -- (3,2.8); 
\draw [<-] (3,3.2) -- (3,4.8);
\draw [->] (7,1.2) -- (7,2.8); 
\draw [<->] (7,3.2) -- (7,4.8);
\draw [->] (9,1.2) -- (9,2.8); 
\draw [<-] (9,3.2) -- (9,4.8);
\draw [<-] (9,5.2) -- (9,6.8);

\draw [->] (1.2,1) -- (2.8,1); 
\draw [<->] (3.2,1) -- (4.8,1); 
\draw [<-] (5.2,1) -- (6.8,1); 
\draw [->] (3.2,5) -- (4.8,5); 
\draw [<-] (5.2,5) -- (6.8,5); 

\draw [->] (4.8,1.2) -- (1.2,2.8); 
\draw [->] (2.8,3.2) -- (1.2,4.8);
\draw [->] (3.15,1.2) -- (4.8,4.8);
\draw [->] (3.25,1.15) -- (6.8,2.8);
\draw [<-] (5,1.2) -- (3.1,4.8);
\draw [<-] (5.2,1.2) -- (6.85,4.8);
\draw [->] (6.75,1.1) -- (3.2,2.8);
\draw [->] (6.85,1.2) -- (5,4.8);
\draw [<-] (3.2,3.2) -- (6.8,4.9);
\draw [->] (3.2,4.85) -- (6.8,3.2);
\draw [<-] (7.2,5.2) -- (8.8,6.8);
\draw [->] (7.2,3.2) -- (8.8,4.8);
\draw [->] (5.2,4.8) -- (8.8,3.2);

\end{tikzpicture}
\caption{The graph $T'$. }
\label{slika5}
\end{center}
\end{figure}

Let  
$\mathcal S= \{\mathcal B_3(e), \mathcal B_5(o),\mathcal B_6(o)\}$ and  $\mathcal R=\{ \mathcal B_3(o),\mathcal B_5(e),\mathcal B_6(e)\}$.
We now assume that  $(3,k),(3,k')\in\mathcal E(\mathcal P)$ and $(3,t)\notin\mathcal E(\mathcal P)$ for $k<t<k'$ . 
We combine everything given in Lemma \ref{paj} above 
 and represent the claims of Lemma \ref{paj}  by a digraph, which we will denote by $T'$ (see Fig.~\ref{slika5}). 

Let us first consider the vertices $\mathcal B_3(e)$ and $\mathcal B_3(o)$ of $T'$. By Lemma \ref{paj}(e), if $\mathcal B_3$ occurs for $k$, then  $\mathcal B_1$ or $\mathcal B_2$ occurs for $k'$ and $k'=k+1$. This explains why there are arcs from 
$\mathcal B_3(e)$ to  $\mathcal B_1(o), \mathcal B_2(o)$, and  from 
$\mathcal B_3(o)$ to  $\mathcal B_1(e), \mathcal B_2(e)$. There 
are no other arcs with starting vertices $\mathcal B_3(e)$ or $\mathcal B_3(o)$.

Suppose that  $\mathcal B_1(o)$ or $\mathcal B_1(e)$ occurs for $k$. By Lemma \ref{paj}(b) and (c), if $\mathcal B_1(o)$ occurs for $k$, then $\mathcal B_3$ occurs for $k'$ and $k'$ is even or $\mathcal B_5$ or $\mathcal B_6$ occurs for $k'$ and $k'$ is odd. The situation is similar if $\mathcal B_1(e)$ occurs for $k$. This explains the arcs $\mathcal B_1(o)\rightarrow \mathcal S$ and $\mathcal B_1(e)\rightarrow \mathcal R.$ There 
are no other arcs with the starting vertices $\mathcal B_1(o)$ or $\mathcal B_1(e)$.

Suppose that  $\mathcal B_2(o)$ or  $\mathcal B_2(e)$ occurs for $k$. By Lemma \ref{paj}(a), $\mathcal B_1,\mathcal B_2,\mathcal B_3,\mathcal B_5,\mathcal B_6$ do not occur for $k'$. Therefore $\mathcal B_4$ or $\mathcal B_7$ occurs for $k'$. We therefore have  arcs from $\mathcal B_2(o)$ and  $\mathcal B_2(e)$  to $\mathcal B_4(e),\mathcal B_4(o),\mathcal B_7(e)$ and $\mathcal B_7(o).$

Suppose that  $\mathcal B_6(e),\mathcal B_6(o),\mathcal B_7(e)$ or $\mathcal B_7(o)$ occurs for $k$.  By Lemma \ref{paj}(b) we have $\mathcal B_1,\mathcal B_2,\mathcal B_4,\mathcal B_7$ do not occur for $k'$. Hence, $\mathcal B_3,\mathcal B_5$ or $\mathcal B_6$ occurs for $k'$. Moreover, by Lemma \ref{paj}(d),  if $\mathcal B_6(o)$ or  $\mathcal B_7(o)$ occurs for $k$, then 
$\mathcal B_3(o), \mathcal B_5(e)$ or $\mathcal B_6(e)$ occurs for $k'$, and if $\mathcal B_6(e)$ or  $\mathcal B_7(e)$ occurs for $k$, then 
$\mathcal B_3(e), \mathcal B_5(o)$ or $\mathcal B_6(o)$ occurs for $k'$. This explains the arcs $\mathcal B_6(o), \mathcal B_7(o)\rightarrow \mathcal R$, and $\mathcal B_6(e),\mathcal B_7(e)\rightarrow \mathcal S$.
There are no other arcs with starting vertices $\mathcal B_6(e),\mathcal B_6(o),\mathcal B_7(e)$ or $\mathcal B_7(o)$.

Suppose that  $\mathcal B_4(e)$ or  $\mathcal B_4(o)$ occurs for $k$. By Lemma  \ref{paj}(e), $\mathcal B_1$ or $\mathcal B_2$ occurs for $k'$ and, if $k$ is even, then $k'$ is odd and vice versa. This explains the arcs 
 $\mathcal B_4(e)\rightarrow\mathcal B_1(o)$, $\mathcal B_4(e)\rightarrow\mathcal B_2(o)$,
$\mathcal B_4(o)\rightarrow\mathcal B_1(e)$ and $\mathcal B_4(o)\rightarrow\mathcal B_2(e).$

Finally, consider the vertices $\mathcal B_5(o)$ or $\mathcal B_5(e)$ that occur for $k$. By Lemma \ref{paj}(a), $\mathcal B_1,\mathcal B_2,\mathcal B_3,\mathcal B_5,\mathcal B_6$ do not occur for $k'$. Hence, $\mathcal B_4$ or $\mathcal B_7$ occurs for $k'$. Therefore there are arcs from  $\mathcal B_5(o)$ and $\mathcal B_5(e)$ to $ \mathcal B_4(e),\mathcal B_4(o),\mathcal B_7(e)$ and $\mathcal B_7(o).$ 

This gives a complete description of 
 graph $T'$  (shown in Fig.~\ref{slika5}).

Let $\mathcal P$ be a path cover of $P_3\Box P_n$ and $\mathcal E(\mathcal P)$ 
be the set of endvertices of paths in $\mathcal P$. Assume that 
 $\mathcal E(\mathcal P)=\{(3,k_1),\ldots,(3,k_r)\}$ and that $k_1<\ldots<k_r$. 
 By Lemma \ref{moznih19} each vertex of $ \mathcal E(\mathcal P)$ corresponds to a vertex 
 of $T'$ and the sequence $(3,k_1),\ldots,(3,k_r)$ corresponds to a directed path in 
 $T'$.
 We denote this directed path of $T'$ by $x_1,\ldots,x_r$. 
 
 Let us define $\mathcal Q=\{\mathcal B_1(e), \mathcal B_6(o), \mathcal B_7(o)\}$.  
 Then, by Lemma \ref{zacetnizadnji}, we have $x_1\in \mathcal S=\{\mathcal B_3(e), \mathcal B_5(o),\mathcal B_6(o)\}$ and $x_r\in \mathcal Q$. 

Suppose that $i$ and $j$ are integers such that $x_i,x_j\in \mathcal S\cup\mathcal  R$ and $x_\ell\notin \mathcal S\cup \mathcal R$ for $i<\ell<j$. 
 Define $X_{i}=\{x_i,\ldots, x_{j-1}\}$ and

\begin{quote} 
$a_i=|X_{i}\cap \{\mathcal B_1(e),\mathcal B_3(e),\mathcal B_5(e),\mathcal B_7(e)\}|$,
$ b_i=|X_i\cap \{\mathcal B_1(o),\mathcal B_3(o),\mathcal B_5(o),\mathcal B_7(o)\}|$,\\
$c_i=|X_i\cap \{\mathcal B_2(e),\mathcal B_4(e),\mathcal B_6(e)\}|$, 
$d_i=|X_i\cap \{\mathcal B_2(o),\mathcal B_4(o),\mathcal B_6(o)\}|$.
\end{quote}
 We will prove  the following  claims.\\


\noindent {\it Claim 1:}  If  $x_i,x_j\in \mathcal S$ and $x_\ell\notin \mathcal S\cup\mathcal  R$ for $i<\ell<j$  then  $a_i+2c_i=b_i+2d_i$.  \\
{\it Claim 2:}  If  $x_i,x_j\in \mathcal R$ and  $x_\ell\notin \mathcal S\cup \mathcal R$ for  $i<\ell<j$, then  $a_i+2c_i=b_i+2d_i$.  \\
{\it Claim 3:} If  $x_i\in \mathcal S,x_j\in \mathcal R$  and $x_\ell\notin \mathcal S\cup \mathcal R$ for $i<\ell<j$, then  $a_i+2c_i+2=b_i+2d_i$. \\
{\it Claim 4:} If  $x_i\in \mathcal R,x_j\in \mathcal S$ and $x_\ell\notin \mathcal S\cup \mathcal R$ for $i<\ell<j$, then  $a_i+2c_i=b_i+2d_i+2$. \\

We will prove the above claims only in the case when  $x_i= \mathcal B_3(e)$;  if $x_i=\mathcal B_5(o) $ or  $x_i=\mathcal B_6(o) $ the proof is similar. 
Suppose therefore that  $x_i= \mathcal B_3(e)$ (see graph $T'$ in Fig.~\ref{slika5}). Then we have either $x_{i+1}=\mathcal B_1(o)$ or $x_{i+1}=\mathcal B_2(o)$.
If $x_{i+1}=\mathcal B_1(o)$,  then $x_{i+2}\in \mathcal S$  and we have case (1) below, otherwise  $x_{i+1}=\mathcal B_2(o)$ and we have case (2) below
\begin{itemize}
\item[(1)] $\mathcal B_3(e) \rightarrow  \mathcal B_1(o) \rightarrow \mathcal S.$
\item[(2)] $\mathcal B_3(e) \rightarrow  \mathcal B_2(o) \rightarrow\underbrace{\mathcal B_4(e) \rightarrow \mathcal B_2(o)\rightarrow \ldots \rightarrow  \mathcal B_4(e) \rightarrow \mathcal B_2(o)}_{\text{some number (possibly 0) of repetitions of } \; \mathcal B_4(e) \rightarrow \mathcal B_2(o)}$ and then one of the following possibilities occurs:

\begin{itemize}
\item[(2.1)] $\mathcal B_4(e)\rightarrow\mathcal B_1(o) \rightarrow \mathcal S.$ 
\item[(2.2)]  $\mathcal B_7(e)\rightarrow \mathcal S.$ 
\item[(2.3)] $\mathcal B_7(o)\rightarrow \mathcal  R.$
\item[(2.4)] $\mathcal B_4(o)\rightarrow \underbrace{\mathcal B_2(e) \rightarrow \mathcal B_4(o)\rightarrow \ldots  \rightarrow \mathcal B_2(e) \rightarrow \mathcal B_4(o).}_{\text{some number (possibly 0) repetitions of }  \mathcal B_2(e) \rightarrow \mathcal B_4(o)}$ In this case we continue as follows: 
\begin{itemize}
\item[(2.4.1)]  $\mathcal B_1(e)\rightarrow \mathcal R.$ 
\item[(2.4.2)] $\mathcal B_2(e)\rightarrow \mathcal B_7(o)\rightarrow \mathcal R.$
\item[(2.4.3)] $\mathcal B_2(e)\rightarrow\mathcal B_7(e) \rightarrow \mathcal S.$ 
\item[(2.4.4)] $\mathcal B_2(e)\rightarrow\mathcal B_4(e)\rightarrow \mathcal B_1(o)\rightarrow \mathcal S.$ 
\item[(2.4.5)] $\mathcal B_2(e)\rightarrow\mathcal B_4(e)\rightarrow \mathcal B_2(o)$ 
and we are back to case (2).

\end{itemize}
\end{itemize}
\end{itemize}

In case (1) $\mathcal B_3(e)$ contributes 1 to $a_i+2c_i$ and $\mathcal B_1(o)$ contributes 1 to $b_i+2d_i$.  Hence, Claim 1 is true in this case.  In case (2), $\mathcal B_3(e)$ contributes 1 to $a_i+2c_i$ and $\mathcal B_2(o)$ contributes 2 to $b_i+2d_i$. 
Note that each repetition of  $\mathcal B_4(e) \rightarrow \mathcal B_2(o)$ contributes 
2 to $a_i+2c_i$ and 2 to $b_i+2d_i$. After  we have arrived at the last occurence  of 
$\mathcal B_2(o)$ in $X_i$, we   analyse the contribution of the remaining vertices to 
$a_i+2c_i$ and to $b_i+2d_i$. 
We claim that for the subcases (2.1) to (2.4) the contribution  of the remaining vertices  to $a_i+2c_i$ is
\begin{itemize}
\item[(a)]   one more than  the contribution to $b_i+2d_i$, if the path ends in $\mathcal S$, and
\item  [(b)]    one less than the contribution to $b_i+2d_i$, if the path ends in $\mathcal R$.
\end{itemize}
Observe that the above claim implies Claim 1 and Claim 3, hence it remains to prove (a) and (b).

In  case (2.1), $\mathcal B_4(e)$ contributes 2 to $a_i+2c_i$ and $\mathcal B_1(o)$ contributes 1 to $b_i+2d_i$, and in case (2.2),  $\mathcal B_7(e)$ contributes 1 to $a_i+2c_i$. In both cases we find that (a) is true. In  case (2.3), $\mathcal B_7(o)$ contributes 1 to $b_i+2d_i$ and therefore (b) is true.  

In case (2.4) $\mathcal B_4(o)$ contributes 2 to $a_i+2c_i$ and each repetition of $\mathcal B_2(e) \rightarrow \mathcal B_4(o)$ contributes 2 to $a_i+2c_i$ and 2 to $b_i+2d_i$. Hence,  the contribution of vertices that appear before  the last occurrence of $\mathcal B_4(o)$ (including the last $\mathcal B_4(o)$), to $a_i+2c_i$  is three less than the contribution  to $b_i+2d_i$. We claim that for the subcases (2.4.1) to (2.4.5) the contribution  of the remaining vertices  to $a_i+2c_i$ is
\begin{itemize}
\item[(c)]  three more than  the contribution to $b_i+2d_i$, if the path ends in $\mathcal S$, and
\item  [(d)]  one more than the contribution to $b_i+2d_i$, if the path ends in $\mathcal R$.
\end{itemize}
Note that (c) and (d) imply (a) and (b), hence it remains to prove (c) and (d). 

In case (2.4.1), $\mathcal B_1(e)$ contributes 1 to $a_i+2c_i$. In case (2.4.2), $\mathcal B_2(e)$ contributes 2 to $a_i+2c_i$ and  $\mathcal B_7(o)$  contributes 1 to $b_i+2d_i$. In both cases we find that (d) is true. In case (2.4.3), $\mathcal B_2(e)$ contributes 2  and  $\mathcal B_7(e)$ contributes 1 to $a_i+2c_i$. In case (2.4.4), $\mathcal B_2(e)$ and $\mathcal B_4(e)$ contribute 4 to $a_i+2c_i$ and  $\mathcal B_1(o)$ contributes 1 to $b_i+2d_i$. In these two cases, (c) is true. 

In case (2.4.5)  we traverse the path (considering the whole path of case (2.4)) 
$$\mathcal B_4(o)\rightarrow  \underbrace{\mathcal B_2(e) \rightarrow \mathcal B_4(o)\rightarrow \ldots  \rightarrow \mathcal B_2(e) \rightarrow \mathcal B_4(o)}_{\text{some number (possibly 0) repetitions of }  \mathcal B_2(e) \rightarrow \mathcal B_4(o)}\rightarrow\mathcal B_2(e)\rightarrow\mathcal B_4(e)\rightarrow \mathcal B_2(o).$$
This path contributes equal to  $a_i+2c_i$ as to $b_i+2d_i$ and it can be traversed several times until eventually  we come to the last occurrence of $B_2(o)$, after which the path $X_i$ ends with one of the cases (2.1),(2.2),(2.3) or (2.4.1) to (2.4.4) which we have dealt with above. This proves Claims 1-4.  

Let   $I=\{i\, |\, x_i\in \mathcal S\cup \mathcal R\}$.
Note that (by an inductive  argument)  Claim 1  remains true also if we omit the condition $x_\ell\notin \mathcal S\cup \mathcal R$, hence if $x_i,x_j\in \mathcal S$ then $$\sum_{\substack{k\in I\\i\leq k< j}}(a_k+2c_k)=\sum_{\substack{k\in I\\i\leq k< j}}(b_k+2d_k).$$ Similarly,  Claims 3 remains true if we omit the condition $x_\ell\notin \mathcal S\cup \mathcal R$. Thus if $x_i\in \mathcal S,x_j\in \mathcal R$ then 
\begin{equation}
\label{SR}
\sum_{\substack{k\in I\\i\leq k< j}}(a_k+2c_k)+2=\sum_{\substack{k\in I\\i\leq k< j}}(b_k+2d_k).
\end{equation}

Note that $x_r\in \mathcal Q=\{\mathcal B_1(e), \mathcal B_6(o), \mathcal B_7(o)\}$, and that  from a vertex of $\mathcal Q$ it is only possible to 
continue (in graph $T$) to a vertex of $\mathcal R$ (see Fig.~\ref{slika5}); so  from $x_r$ it is also only possible to continue to a vertex of $\mathcal R$.  For $i=1$ and $j=r=\max\{i\,|\,i\in I\}$   the left side of \eqref{SR} is equal to $|\mathcal V_e|+2|\overline{\mathcal V}_e|+2$  and the right side of \eqref{SR} is equal to $|\mathcal V_o|+2|\overline{\mathcal V}_o|$. Hence we have  
$|\mathcal V_e|+2|\overline{\mathcal V}_e|+2=|\mathcal V_o|+2|\overline{\mathcal V}_o|$, which completes the proof of Lemma \ref{lemarob}.

\end{document}